\documentclass[11pt]{amsart}
\usepackage{etex}
\usepackage{diagbox}
\usepackage{array}
\usepackage{graphicx, pictex}
\usepackage{accents}
\usepackage{stmaryrd}
\usepackage{changepage}
\usepackage{float}
\restylefloat{table}
\usepackage{multirow}
\usepackage[dvipsnames]{xcolor}
\usepackage{amsmath,amssymb,amsthm,mathrsfs}
\allowdisplaybreaks
\usepackage{yfonts}
\usepackage{tikz}
\usepackage{fnpct}

\renewcommand{\d}{\partial}

\def\tildeu#1{\underaccent{\sim}{#1}}

\def\d{\Omega}

\def\Forall{\quad \hbox{ for all }}

\def\M{{\mathcal{M}}}

\newcommand{\tn}[1]{\lVert\kern-1pt\lvert{#1}\rvert\kern-1pt\rVert}
\def\<{{\langle}}
\def\>{{\rangle}}
\def\Forall{\quad \hbox{ for all }}

\def\d{\Omega}

\def\Forall{\quad \hbox{ for all }}

\def\d{\Omega}

\def\Forall{\quad \hbox{ for all }}

\def\tb#1{{\|\kern-1pt| #1 \|\kern-1pt|}}
\def\nm2#1#2{\|#1\|_{2,\d_{#2}}}

\def\R{\mathbb{R}}

 \theoremstyle{plain}
 \newtheorem{thm}{Theorem}[section]
 \numberwithin{equation}{section} 
 \numberwithin{figure}{section} 
 \theoremstyle{plain}
 
 \theoremstyle{plain}
 \theoremstyle{plain}
 \newtheorem{theorem}[thm]{Theorem}
 \theoremstyle{plain}
 \newtheorem{corollary}[thm]{Corollary}
\theoremstyle{plain}
 \newtheorem{remark}[thm]{Remark}
 \theoremstyle{plain}
 \newtheorem{lemma}[thm]{Lemma}

\def\M{{\mathcal{M}}}

\def\d{{\Omega}}

\def\Forall{\quad \hbox{ for all }}

\def\<{{\langle}}
\def\>{{\rangle}}
\def\R{\mathbb{R}}

\usepackage{amssymb}

\begin{document}

\title[Optimal convergence of  upwinding Petrov Galerkin]
{On convergence of  upwinding Petrov-Galerkin methods for  convection-diffusion}

\author{Constantin Bacuta}
\address{University of Delaware,
Mathematical Sciences,
501 Ewing Hall, 19716}
\email{bacuta@udel.edu}


\keywords{Mixed finite element, Upwinding Petrov-Galerkin, bubble upwinding, convection dominated problem, singularly perturbed problems}

\subjclass[2010]{65F, 65H10, 65N06, 65N12, 65N22, 65N30, 74S05, 76R10, 76R50}


\begin{abstract} We consider special  upwinding Petrov-Galerkin discretizations  for  convection-diffusion problems.  For the one dimensional case with  a standard continuous linear element  as  the trial space and a special exponential bubble test space,  we prove  that the Green function associated to the continuous solution can  generate the  test space.  In this case, we find  a formula for the exact inverse of the  discretization matrix, that  is used  for establishing  new error estimates for other  bubble upwinding Petrov-Galerkin discretizations.  We introduce a quadratic bubble upwinding  method with a special  scaling parameter  that  provides optimal approximation order  for the solution in the discrete infinity norm. 
Provided the  linear interpolant has standard  approximation properties, we prove optimal approximation estimates  in $L^2$ and $H^1$ norms. The  quadratic bubble method is  extended to a two dimensional convection diffusion problem.  The proposed discretization produces optimal  $L^2$ and $H^1$ convergence  orders on subdomains that avoid the boundary layers.  

The tensor idea of using an efficient upwinding Petrov-Galerkin discretization along each stream line direction in combination  with a standard discretizations for  the orthogonal  direction(s) can  lead to new and efficient discretization methods for multidimensional convection dominated models.

\end{abstract}
\maketitle

\section{Introduction}


We consider   the singularly perturbed  convection-diffusion problem: Find $u$ defined on $\Omega$ such that 
\begin{equation}\label{PDE_RD}
   \left\{
\begin{array}{rcl}
     -\varepsilon\, \Delta u +\textbf{b}\cdot \nabla u =\ f & \mbox{in} \ \ \ \Omega,\\
      u =\ 0 & \mbox{on} \ \partial\Omega,\\ 
\end{array} 
\right. 
\end{equation}
for $\varepsilon>0$, $\textbf{b}=(b_1, b_2)^T \neq 0$,  and $b_1\geq 0, b_2\geq  0$ on $\Omega=(0, 1)\times (0,1) \subset\R^2$. 
We focus on the convection dominated case, i.e.,  $\varepsilon \ll 1$ and  $f \in L^2(\Omega)$.

The one dimensional version of \eqref{PDE_RD} with $\textbf{b}=1$  is: 
Find $u=u(x)$  on $[0, 1]$ such that
\begin{equation}\label{eq:1d-model}
\begin{cases}-\varepsilon\,  u''(x)+u'(x)=f(x),& 0<x<1\\ u(0)=0, \ u(1)=0. \end{cases}
\end{equation} 

The PDE models \eqref{PDE_RD},  \eqref{eq:1d-model}, and their  various multi-dimensional  extensions   arise in solving various  practical problems such as  heat transfer problems in thin domains, as well as when using small step sizes in implicit time discretizations of parabolic reaction diffusion type problems, see e.g., \cite{Lin-Stynes12} and the references in  \cite{roos-stynes-tobiska-96}. The  discretization of these types of  problems  poses  numerical challenges due to the $\varepsilon$-dependence of resulting linear systems and of  the  stability constants. In addition, the solutions to these problems are characterized by  boundary layers, and are difficult to approximate using a standard finite element space, see e.g., 
\cite{Dahmen-Welper-Cohen12, EEHJ96,  linssT10, Roos-Schopf15, zienkiewicz2014}. 
There is a tremendous amount of literature addressing  these types of problems, see e.g.,  \cite{EEHJ96, linssT10, quarteroni-sacco-saleri07, Roos-Schopf15, zienkiewicz2014}. Our approach on building the test spaces and on establishing  error analysis is different, as we seek test spaces that lead to optimal discrete infinity error  first, and then analyze  standard norm errors. The proposed strategy leads to  more efficient discretizations  for convection dominated problems. 

In this paper, we discuss  a general upwinding Finite Element (FE) Petrov-Galerkin (PG) method, called Upwinding Petrov-Galerkin (UPG)  discretization based on  bubble modification of the test space. This approach works for both 1D and 2D cases. While the trial space is a standard continuous $P^1$- FE space, the test space is obtained by modifying each basis function of the trial space using two bubble functions with  special average values, and alternating signs to match  the convection direction. 
For  many classical discretization methods for  convection diffusion problems, including  the streamline diffusion (SD) method, the error analysis is done using an $\varepsilon$-weighted norm, and not studied in standard  $L^2$ or  $H^1$ norms. 
In the proposed approach, we  design our discretization such that  the discrete solution is close to the interpolant of the exact solution. Such discretization  avoids non-physical oscillations  and allows for optimal   $L^2$ and $H^1$ convergence  estimates. 

The goal of the paper is to present  new techniques and ideas for robust  finite element discretization  and analysis of convection dominated problems based on special bubble UPG approach. The  ideas and techniques presented here, can be extended  to the multidimensional case for  other  convection dominated problems. 

The rest of the paper is organized as follows.  In Section \ref{sec:FE}, we first  review a general upwinding Petrov-Galerkin discretization method, and  define particular test spaces based on quadratic bubbles and exponential type bubbles. In Section  \ref{sec:Exact-Inv4EB},  we establish an exact inverse of the exponential bubble UPG discretization matrix. 
  We  analyze the convergence  of  a special  quadratic bubble UPG  approximation in Section \ref{sec:B2}. We prove optimal   $L^2$ and $H^1$ convergence  estimates for the discrete solution in Section \ref{sec:B2norms}. In Section \ref{sec:2DB2}, we extend the special  quadratic bubble UPG  approximation to the two dimensional case. 
  We present numerical results in Section \ref{sec:2DB2NR},  and summarize our findings  in Section \ref{sec:Conc}.

\section{Finite element  linear variational formulation}\label{sec:FE}
For the finite element discretization of \eqref{eq:1d-model}, we use the following  notation:
\[ 
\begin{aligned}
a_0(u, v) & = \int_0^1 u'(x) v'(x) \, dx, \ (f, v) = \int_0^1  f(x) v(x) \, dx,\ \text{and}\\
b(v, u)& =\varepsilon\, a_0(u, v)+(u',v)  \ \text{for all} \ u,v \in V=Q=H^{1}_0(0,1).
\end{aligned}
 \]
A variational formulation of \eqref{eq:1d-model}  is:  Find $u \in Q:= H_0^1(0,1)$ such that
 \begin{equation}\label{VF1d}
b(v,u) = (f, v), \ \text{for all} \ v \in V=H^{1}_0(0,1).
\end{equation}
The existence and uniqueness of the solution of \eqref{VF1d} is well known, see e.g., \cite{BQ15, BQ17, bartels16, boffi-brezzi-fortin, braess, brenner-scott, demko23, ern-guermond, demko-oden}.

\subsection{The Petrov-Galerkin method  with bubble type test space } \label{sec:PG}
 Various  Petrov-Galerkin discretizations  for solving \eqref {VF1d} were considered in  \cite{CRD-results, zienkiewicz76, mitchell-griffiths, roos-stynes-tobiska-96, zienkiewicz2014} and other related papers. According to Section 2.2.2 in \cite{zienkiewicz2014}, the idea  of upwinding with  polynomial bubble functions was first suggested in \cite{ZGHupwind76} and used with quadratic bubble functions modification in the same year in \cite{zienkiewicz76}.    

 In this section, following \cite{connections4CD, comparison4CD}, we review  a general class of upwinding PG discretizations based on a bubble modification of the  standard $C^0-P^1$ test space $\M_h$.  
The idea is to define the test space  $V_h$  by adding  a pair of  bubble functions to each basis function  
$\varphi_j \in \M_h$ in order to match the convection direction.

For defining the general bubble UPG discretization of  \eqref{eq:1d-model}, we  start by  dividing the interval $[0,1]$ into $n$  equal length subintervals using the nodes $0=x_0<x_1<\cdots < x_n=1$, and denote  $h:=x_j - x_{j-1}=1/ n$. We define  the corresponding finite element discrete space  $\M_h$  as  the subspace of $Q=H^1_0(0,1)$, given by
 \[ 
 \M_h = \{ v_h \in V \mid v_h \text{ is linear on each subinterval}\  [x_j, x_{j + 1}]\},
 \]
  i.e., $\M_h$ is the space of all {\it continuous piecewise linear  functions} with respect to the given nodes, that {\it are zero at $x=0$ and $x=1$}.  We consider the nodal basis $\{ \varphi_j\}_{j = 1}^{n-1}$ with the standard defining property $\varphi_i(x_j ) = \delta_{ij}$.  

For the trial space $\M_h \subset Q=H^1_0(0,1)$, a  general  Petrov-Galerkin method for solving \eqref{VF1d} chooses  a test space $V_h \subset V=H^{1}_0(0,1)$ that is in general different from $\M_h$. To review  the bubble  UPG method, we first consider a continuous  bubble generating function $B:[0,h] \to\R$ with the  properties:
\begin{equation}\label{Bbounds}
B(0)=B(h)=0,\\
\end{equation}

\begin{equation}\label{b1}
\frac{1}{h}\, \int_0^h B(x)  \, dx=b \ \text{with} \ b>0.\\
\end{equation}

 Next, for $ i=1,2, \cdots, n$, we generate $n$ locally supported bubble functions  by translating $B$. We define $B_i:[0, 1] \to \R$ by $B_i(x)=B(x-x_{i-1})$ on $[x_{i-1}, x_i]$, and we extend it  by zero to the entire interval $[0, 1]$. 
 
 The {\it bubble upwinding}  idea is based on building $V_h$  such that diffusion is created from  multiplying the convection term with special test functions. 
 We define  $g_j:=\varphi_j  + B_{j}-B_{j+1}$ and  the test space $V_h$ by 
 \[
 V_h:= span \{g_j\}_{j = 1}^{n-1} =span \{\varphi_j  + B_{j}-B_{j+1} \}_{j = 1}^{n-1}.
 \]
We note that both $\M_h$ and $V_h$ have the same dimension of $(n-1)$. 

The bubble UPG discretization of 
\eqref{eq:1d-model} is: Find $u_h \in \M_h$ such that 
\begin{equation}\label{eq:1d-modelPG}
b(v_h, u_h) = \varepsilon\, a_0(u_h, v_h)+(u'_h,v_h) =(f,v_h) \ \Forall  v_h \in V_h. 
\end{equation}
The variatonal formulation \eqref{eq:1d-modelPG} admits a reformulation that uses only  {\it standard  linear finite element spaces}.  To describe the reformulation, we assume 
\[
u_h= \sum_{j=1}^{n-1} \alpha_j \varphi_j,
\]
and consider a generic {\it test function} 
\[
v_h= \sum_{i=1}^{n-1} \beta_i \varphi_i + \sum_{i=1}^{n-1}  \beta_i (B_i - B_{i+1}) = \sum_{i=1}^{n-1} \beta_i \varphi_i + \sum_{i=1}^{n}  (\beta_i - \beta_{i-1}) B_{i},
\]
where, we define $\beta_0=\beta_n=0$. By introducing the notation 
\[
B_h:=\sum_{i=1}^{n}  (\beta_i - \beta_{i-1}) B_{i},  \ \text{and } \  w_h:=  \sum_{i=1}^{n-1} \beta_i \varphi_i ,
\]
we get 
$v_h=w_h + B_h$. As presented in detail in \cite{connections4CD, comparison4CD}, 
for any $u_h \in \M_h$ and  $v_h=w_h + B_h \in V_h$, we get
\begin{equation} \label{eq:bPGr}
 b(v_h, u_h) = \left (\varepsilon + bh\right )  (u'_h, w'_h) +  (u'_h, w_h).
\end{equation}
 The addition of the bubble part to the test space  leads to the extra diffusion term $bh  (u'_h, w'_h)$ with  $bh >0$  matching the sign of the coefficient of $u'$ in  \eqref{eq:1d-model}. This technique justifies the terminology  of 
 {\it ``upwinding PG''} method. 
 
Note that only the linear part $w_h$ of $v_h$ appears in the expression of $ b(v_h, u_h)$ of \eqref{eq:bPGr}, and  the  functional  $v_h \to (f, v_h)$ can  be also viewed as a functional  of the linear part  of  $v_h\in V_h$. Indeed, Using the splitting $v_h=w_h + B_h $,  we  have
\[
(f, v_h) =(f, w_h) + (f, \sum_{i=1}^{n}  h w'_h B_i) =(f, w_h) +h\,  (f, w'_h  \sum_{i=1}^{n}   B_i). 
\]
As a consequence, the variational formulation of the upwinding  Petrov-Galerkin method can be reformulated as:   Find $u_h \in \M_h$ such that 
\begin{equation}\label{eq:1d-modelPGR}
 \left (\varepsilon + bh\right )  (u'_h, w'_h) +  (u'_h, w_h) = (f, w_h )  + h\,  (f, w'_h  \sum_{i=1}^{n}   B_i),  w_h \in M_h. 
\end{equation}

 We note that  the reformulation \eqref{eq:1d-modelPGR}   uses  the same  space of piecewise linear functions for the test space and for the  trial space. 
The diffusion coefficient of $(u'_h, w'_h)$, in \eqref{eq:1d-modelPGR} is  now $\varepsilon +h\, b$ and the corresponding bilinear form  
 is coercive. Thus, \eqref{eq:1d-modelPGR} has unique solution $u_h$ and consequently $u_h$ is the unique solution of   the bubble UPG discretization \eqref{eq:bPGr}.

The reformulation \eqref{eq:1d-modelPGR} also leads  to the  linear system 
\begin{equation}\label{1d-PG-ls}
\left ( \left (\frac{\varepsilon}{h} + b \right ) S+ C \right )\, U = F_{PG}, 
\end{equation}
where \(S,C \in\R^{n-1}\times \R^{n-1}\)  are tridiagonal matrices:
\[
S=tridiag(-1, 2, -1),   \ C= tridiag\left (-\frac12, 0, \frac12 \right ), 
 \]
 and  the vectors \(U,F_{PG}\in\R^{n-1}\)  are defined by
 \[
 U:=\begin{bmatrix}u_1\\u_2\\\vdots\\u_{n-1}\end{bmatrix},\quad F_{PG}:= \begin{bmatrix}(f,g_1)\\ (f,g_2)\\\vdots \\ (f,g_{n-1})\end{bmatrix}. 
\]

We  note that  the matrix of the finite element  system  \eqref{1d-PG-ls} is 
\begin{equation} \label{eq:M4CDfe} 
M_{fe}= tridiag\left ( -\left (\frac{\varepsilon}{h} +b \right ) - \frac{1}{2},\  2\, \left (\frac{\varepsilon}{h} +b \right ),\  -\left (\frac{\varepsilon}{h} +b \right )+ \frac{1}{2} \right ),
\end{equation}
and the matrix $M_{fe}$ depends only on  $\varepsilon, h$ and the average value $b$ of the generating bubble $B$.
 More precisely, it depends only on $\frac{\varepsilon}{h} +b$. 

\subsection{Upwinding PG with quadratic bubble functions} \label{sec:Quad-Bubbles} 

In this section, we review   a quadratic bubble UPG  for the model problem \eqref{VF1d}, that was also discussed in e.g.,  \cite{CRD-results,  zienkiewicz76, mitchell-griffiths, ZGHupwind76, zienkiewicz2014}.  In the next section, we show that choosing a special scaling parameter for  the generating bubble, the method has high order of approximation in the discrete infinity norm. 

The discrete trial space is   
 $\M_h= span\{ \varphi_j\}_{j = 1}^{n-1}$, as above.  The test space $V_h$ is a modification of $\M_h$, using {\it quadratic bubble functions}. Here are the details.
 
   First, for a parameter $\beta>0$,  we define the bubble function $B$ on $[0, h]$ by
 \begin{equation} \label{eq:Bq}
B^q(x)= B(x)= \frac{4\, \beta}{h^2} x(h-x).
 \end{equation}
 Elementary calculations show that \eqref{b1} 
 holds with $b=\frac{2\, \beta}{3}$. 
 Using the function $B$ and the general construction of Section \ref{sec:PG}, we define the set of bubble functions $\{B^q_1, B^q_2, \cdots,B^q_n\}$  on $[0, 1]$,  where $B^q_i(x)=B^q(x-x_{i-1})$, and 
 \[
 V_h:= span \{ \varphi_j + (B^q_{j}-B^q_{j+1}) \}_{j= 1}^{n-1}.
 \]
  In this case,  $b= \frac{2\beta}{3}$, and according to \eqref{eq:M4CDfe},  we obtain 
 \[
M^q_{fe}  =  tridiag\left ( -\left ( \frac{\varepsilon}{h} + \frac{2\beta}{3}\right )- \frac{1}{2},\  2 \left ( \frac{\varepsilon}{h} + \frac{2\beta}{3}\right ) ,\   -\left ( \frac{\varepsilon}{h} + \frac{2\beta}{3}\right )  + \frac{1}{2} \right ). 
\]
\subsection{Upwinding PG with exponential bubble functions} \label{sec:Exponential-Bubbles} 
We review  an exponential  bubble UPG  for the model problem \eqref{VF1d}. For more details, see   \cite{connections4CD, comparison4CD}.  The discrete trial  space space  is the same 
 $\M_h= span\{ \varphi_j\}_{j = 1}^{n-1}$. The discrete test space $V_h$ is a modification of $\M_h$ by using an {\it exponential  bubble function}. We define the bubble function $B$ on $[0, h]$ as the solution of the following boundary value problem
 \begin{equation}\label{eq:expB}
 -\varepsilon B'' -B' =1/h, \ B(0)=B(h)=0.
 \end{equation}
  Using the function $B^e=B$ and the general construction of Section \ref{sec:PG}, we define the set of bubble functions $\{B^e_1, B^e_2, \cdots,B^e_n\}$  on $[0, 1]$, where \\ $B^e_i(x)=B^e(x-x_{i-1})$,  and 
 \begin{equation}\label{eq:VhE}
 V_h:= span \{ \varphi_j + (B^e_{j}-B^e_{j+1}) \}_{j= 1}^{n-1}=span \{ g_j \}_{j= 1}^{n-1},
  \end{equation}
with $ g_j:=\varphi_j + (B^e_{j}-B^e_{j+1})$, $j=1,2,\cdots,n-1$.

 It is easy to check that the unique solution of \eqref{eq:expB} is 
 \begin{equation}\label{eq:Be}
B^e(x)= B(x)=\frac{1 - e^{-\frac{x}{\varepsilon}}}{1 - e^{-\frac{h}{\varepsilon}}} - \frac{x}{h}, \ x \in [0, h],  \ \text{and} 
  \end{equation}
   \begin{equation}\label{IntexpB}
  \frac{1}{h}\,  \int_0^{h} B(x)\, dx = \frac{1}{2 t_e} - \frac{\varepsilon}{h}, \ \text{where} 
     \end{equation}
  \begin{equation}\label{eq:g0}
 t_e:=\tanh\left (\frac{h}{2\varepsilon}\right)= \frac{ e^{\frac{h}{2\varepsilon}} -e^{-\frac{h}{2\varepsilon}}} {e^{\frac{h}{2\varepsilon}} +e^{-\frac{h}{2\varepsilon}}}= \frac  {1- e^{-\frac{h}{\varepsilon}} }{1+e^{-\frac{h}{\varepsilon}} }.
  \end{equation}
Consequently,  we have that \eqref{b1} holds with $b= \frac{1}{2 t_e} - \frac{\varepsilon}{h}$, 
  and using  \eqref{eq:M4CDfe}, we obtain that  the matrix for the UPG finite element discretization with exponential bubble test space becomes
  \begin{equation} \label{eq:M4CDfeEB}
 M_{fe}^e = tridiag\left ( -\frac{1+t_e}{2t_e} ,\  \frac{1}{t_e},\   -\frac{1-t_e}{2t_e} \right ). \end{equation}
The upwinding PG method based on the exponential bubble produces in fact the exact solution at the nodes. 
Variants of this result are known in various forms, see e.g., \cite{roos-stynes-tobiska-96, Roos-Schopf15}.  A detailed proof based on the UPG construction presented in section \ref{sec:PG},  can be found in \cite{connections4CD}. Next,  we include our version of the  result  and the main proof  ideas. 
%
\begin{theorem}\label{uFE} Let $\displaystyle u_h := \sum_{i=1}^{n-1} u_i \varphi_i$  be the finite element solution of  \eqref{eq:1d-modelPG}  with the test space as defined in \eqref{eq:VhE}. Then $u_h$ coincides with the linear interpolant $I_h(u)$ of  the exact solution $u$ of the problem  \eqref{eq:1d-model} on the nodes $x_0,x_1, \ldots, x_n$.  Equivalently, we have that 
\[
u_j=u(x_j),\  j=1,2,\cdots,n-1.
\]
\end{theorem}
\begin{proof}

For  $j=1,2, \cdots,n-1$, we multiply the differential equation of  \eqref{eq:1d-model} by $g_j$ and integrate by parts to  obtain that   
\begin{equation} \label{eq:IPde4}
-\frac{1+t_e}{2\, t_e} \,  u(x_{j-1}) + \frac{1}{t_e} \,  u(x_{j})  - \frac{1-t_e}{2\, t_e} \,  u(x_{j+1})=(f, g_j).
\end{equation}
Thus,  the matrix of the system \eqref{eq:IPde4}  with  ``unknown''  vector \\ $U_e=[u(x_1),\cdots, u(x_{n-1})]^T$,
coincides with the matrix  of the system  \eqref{1d-PG-ls}  with $b= \frac{1}{2 t_e} - \frac{\varepsilon}{h}$, i.e., the matrix  $M_{fe}^e$  of \eqref{eq:M4CDfeEB} for the exponential UPG discretization. 
Since $M_{fe}^e$ is invertible, and the right hand sides of the two systems  coincide, 
we can conclude that $u_j=u(x_j)$, $j=1,2,\cdots,n-1$.
\end{proof}

\section{The exact inverse of the exponential buubble UPG  matrix} \label{sec:Exact-Inv4EB} 
In this section, we find  a very  useful formula for the inverse of the matrix $M_{fe}^e$  associated with the exponential UPG discretization.  

Assuming  that $f$ is continuous on $[0, 1]$, using the Green's function for the  problem  \eqref{eq:1d-model},  we have that the solution $u$ satisfies: 
\begin{equation}\label{eq:GreensFnc}
	u(x) = \int_0^1 G(x,s)f(s)\,ds, 
\end{equation}
where $G(x,s)$ can be explicitly determined by using standard integration arguments, and 
  \[ 
 G(x,s)=\frac{1}{e^\frac{1}{\varepsilon}-1} \begin{cases}(e^\frac{1}{\varepsilon}-e^\frac{x}{\varepsilon})(1-e^{-\frac{s}{\varepsilon}}),&0\leq s < x\\(e^\frac{x}{\varepsilon}-1)(e^\frac{1-s}{\varepsilon}-1),&x\leq s\leq 1.\end{cases}
 \] 
 With the notation of Section \ref{sec:Exponential-Bubbles}, we will prove that the entries of the inverse  matrix of the UPG with exponential bubble discretization, defined by  \eqref{eq:M4CDfeEB}, can be described by evaluations of the Green function at the cross-grid of the  interior nodes. The result allows for comparison of the exponential bubble UPG method  with other bubble UPG methods, such as the quadratic bubble UPG. 
 In order to establish the formula, we  prove the following lemma first.
 \begin{lemma}
 For any inside node $x_j=h\,j \in (0, 1) $, the function defined by  $s \to G(x_j,s)$ belongs to test space    
$ V_h  = \text{span}\{g_i\}=  \text{span}\{\varphi_i + B^e_i-B^e_{i+1}\}$, and
 
 \begin{equation}\label{eq:Gxj}
 G(x_j,s)  = \sum_{i=1}^{n-1} G(x_j, x_i) \, g_i(s), \ \text{on}\ [0, 1].
\end{equation}
 \end{lemma} 
 \begin{proof}
 First, we mention that,  for  $i=1,2, \cdots, n-1$,  the  test functions $g_i$ is supported in $[x_{i-1}, x_{i+1}]$  and    
\begin{equation}\label{eq:gj}
 g_i =  \left\{
 \begin{array}{ccl}
    \displaystyle  B^e_i +   \varphi_i  & \mbox{if } \  x \in [x_{i-1}, x_i],\\ \\
    \displaystyle  - B^e_{i+1} +   \varphi_i  & \mbox{if } \  x \in [x_i, x_{i+1}].
\end{array} 
\right.
\end{equation}
To justify \eqref{eq:Gxj}, it would be enough to show that on each interval $[x_{i-1}, x_i]$, we have 
\begin{equation}\label{eq:gj2}
G(x_j,s)_{|_{[x_{i-1}, x_i]}} = G(x_j,x_{i-1}){g_{i-1}}_{|_{[x_{i-1}, x_i]}}  + G(x_j,x_i)\, {g_i}_{|_{[x_{i-1}, x_i]}}. 
\end{equation}
Based on  \eqref{eq:gj}, the identity \eqref{eq:gj2} is equivalent  with showing that  on each interval $[x_{i-1}, x_i]$,
\begin{equation}\label{eq:gj3}
G(x_j,s) - G(x_j,x_{i-1}) = g_i \, \left (G(x_j,x_{i}) -G(x_j,x_{i-1} \right ). 
\end{equation}
This can be easily  verified  by considering the  cases:  I) $j\geq i$ and II)  $j < i$. 
 \end{proof}
 Next, we define the $(n-1)\times (n-1)$ Green matrix $G^m$ with the entries 

  \begin{equation} \label{eq:Gm}
 G^m_{j, i} = G(x_j, x_i), \ \text{where} \ x_j=h\,j, j=1,2,\cdots, n-1.
\end{equation}
 Now we are ready to state the main result of this section:
 
\begin{theorem}\label{th:InvMfe} The inverse  of the exponential bubble UPG finite element discretization matrix  \eqref{eq:M4CDfeEB} is given by

  \begin{equation} \label{eq:InvMfe}
(M^e_{fe})^{-1} = G^m. 
\end{equation}
\end{theorem}
 \begin{proof}
 
 Using the Green's formula for finding the exact solution  $u$  of \eqref{eq:1d-model} at  $x=x_j$, and the identity \eqref{eq:Gxj}, we have 
 \begin{equation} \label{eq:InvMfe2}
u(x_j)  = \int_0^1 G(x_j,s) f(s)\, ds = \sum_{i=1}^{n-1} G(x_j, x_i)  (f,g_i).
\end{equation}
Thus, the vector  $U_e=[u(x_1),\cdots, u(x_{n-1})]^T$ satisfies 
\[
U_e= G^m\, \tildeu{f},
\]
 where $\tildeu{f}:=[(f,g_1), \cdots, (f,g_{n-1}) ]^T$. 
 On the other hand, from \eqref{eq:IPde4} and \eqref{eq:M4CDfeEB}, we obtain 
\[
U_e= (M^e_{fe})^{-1} \, \tildeu{f}. 
\]
Since both  identities hold for all  continuous functions $f$, we can confirm the validity of \eqref{eq:InvMfe}.
\end{proof}
\begin{remark}\label{re:magicB} The result of Theorem \ref{th:InvMfe} holds for the matrix $M^e_{fe}$  of the exponential bubble UPG discretization. However, the formula  \eqref{eq:M4CDfe}  for the general UPG discreization matrix depends only on $\varepsilon$, $h$, and the average $b$ of the generating bubble $B$. Thus, we can rescale the bubble $B$ to have the same average as
the average of the  exponential generating bubble $B^e$. \\  In this case, we will  also have $(M_{fe})^{-1} = (M^e_{fe})^{-1} =G^m$. As an example of such bubble scaling procedure, we consider  the quadratic bubble UPG with the special choice of $\beta$ such that 
\begin{equation}\label{eq:magicB}
b=\frac{2\, \beta}{3} =  \frac{1}{2 t_e} - \frac{\varepsilon}{h}, \ \text{or} 
\ \beta = \frac{3}{4} \left ( \frac{1}{\tanh \left(\frac{h}{2\,\varepsilon}\right )} -
 \frac{2\,\varepsilon}{h} \right ), 
\end{equation}
 and obtain that 
\[
M^q_{fe} = M^e_{fe} \ \text{and consequently}, \ (M^q_{fe})^{-1} = (M^e_{fe})^{-1} =G^m.
\]
In the next section, we see that the special  choice for the  scaling parameter $\beta$  helps with finding  a sharp estimate for the discrete infinity error of the  bubble UPG method. 

\end{remark}


\section{The  quadratic bubble UPG  approximation} \label{sec:B2} 

When discretizing with the exponential bubble UPG, we note that \\  $B^e(x)\approx 1- \frac{x}{h}$ for $\varepsilon \ll h$. In computations, such approximation occurs often  due to  the rounding  error in the double precision arithmetic. For example, $1 \pm e^{-36.05}$ is computed as $1$ in double precision arithmetic. Thus, whenever $h \geq 36.05\,  \varepsilon$, the function  $B^e(x)$ is identical to  $1- \frac{x}{h}$ from the computational point of view. Consequently, the error in computing  the dual vector  could lead to significant  errors in estimating the  discrete solutions. When the exact solution is available,  there are cases when the  quadratic bubble UPG method, with special scaling $\beta$, performs better than the exponential  bubble UPG method. In this section, we will analyze the convergence  of the  quadratic bubble UPG  approximation with the special scaling $\beta$.
 \begin{theorem} \label{th:T}  Let  $f\in C^1([0, 1])$, and let  $u$  be the solution of  the  problem  \eqref{eq:1d-model}.   Let  $\displaystyle u_h=\sum_{i=1}^{n-1} u_i \varphi_i$ be  the solution of  quadratic bubble UPG discretization  \eqref{eq:1d-modelPG} with  scaling parameter $\beta$ given by \eqref{eq:magicB}. Assume  that, for a given $\varepsilon \ll 1$, the mesh size $h$ is chosen such that $e^{-\frac{h}{\varepsilon}} \leq h$. Then, 
 \begin{equation}\label{eq:ConvB2}
 \max_{j=\overline{1,n-1}}\  \left | u(x_j)-u_j \right | \leq  6\, \varepsilon \, \|f\|_\infty + \frac{3}{4}\, h^2\, \|f'\|_\infty . 
 \end{equation}
 \end{theorem}
 \begin{proof}  Using Theorem \ref{uFE}, Theorem \ref{th:InvMfe}, and Remark \ref{re:magicB}, we have that  the vectors $U_e=[u(x_1),\cdots u(x_{n-1})]^T$ and $U_q=[u_1,\cdots, u_{n-1}]^T$  satisfy 
\[
U_e= G^m\, \tildeu{f}^e, \and \ \ U_q= G^m\, \tildeu{f}^q
\]
 where 
 \[
 \begin{aligned}
 \tildeu{f}^e &=[(f,\varphi_1 +(B^e_1-B^e_2)), \cdots, (f,\varphi_{n-1} + (B^e_{n-1}-B^e_n)) ]^T,   \\
 \tildeu{f}^q &=[(f,\varphi_1 +(B^q_1-B^q_2)), \cdots, (f,\varphi_{n-1} + (B^q_{n-1}-B^q_n)) ]^T, \ \text{and} 
 \end{aligned}
 \]
 $G^m$ is the Green matrix defined in \eqref{eq:Gm}.
 Thus, for $j=1,2,\cdots, n-1$,
 \[
 u(x_j)-u_j = \sum_{i=1}^{n-1} G(x_j,x_i) \left (f, B^e_i-B^q_i -(B^e_{i+1}-B^q_{i+1}) \right ).
 \]
 Introducing the notation $B^d_i:=B^e_i-B^q_i$, $i=1,2,\cdots, n$ we have
 \[
 u(x_j)-u_j = \sum_{i=1}^{n-1} G(x_j,x_i) \left (f, B^d_i-B^d_{i+1}  \right ).
 \] 

 Next, we estimate  $u(x_j)-u_j$ by using summation by parts  with respect to the $i$-index in the right hand side of the above  sum. Thus,  we have 
\begin{equation}\label{eq:dif1}
 u(x_j)-u_j = \sum_{i=1}^{n} \left ( G(x_j,x_i)  -G(x_j,x_{i-1})  \right )\,  (f, B^d_i ),
 \end{equation}
 
where 
\[
G(x_j,x_0) =G(x_j,x_n) =0, \ \text{for}\  j=1,2,\cdots, n-1.
\]

 For a fixed  $j=1,2,\cdots, n-1$,  it is not difficult  to check that 
\[ 
| G(x_j,x_i)  -G(x_j,x_{i-1})| <
 \begin{cases} 1 , & \text{for} \  i=1 \ \text{and} \ i=j+1,\\ 
 e^{-\frac{h}{\varepsilon}} ,&  \text{for} \   i \neq 1\  \text{or} \ i \neq j+1.
 \end{cases}
 \]
 Consequently, from \eqref{eq:dif1}, we have
 \[
 \begin{aligned}
  |u(x_j)-u_j| & \leq  |(f, B^d_1)| + |(f, B^d_{j+1})| +  \sum_{ i \neq 1, i \neq j+1}  e^{-\frac{h}{\varepsilon}}\, 
   |(f, B^d_i )| \\
   & \leq  \max_{i=\overline{1,n}}\,  |(f, B^d_i )|  \left (2+ (n-2) e^{-\frac{h}{\varepsilon}}\right ).
   \end{aligned}
 \]
 Using the assumption  $e^{-\frac{h}{\varepsilon}} \leq h =1/n$, we obtain 
 
 \begin{equation}\label{eq:dif2}
 | u(x_j)-u_j | \leq 3\,  \max_{i=\overline{1,n}}\,  |(f, B^d_i )|. 
 \end{equation}
To estimate $ |(f, B^d_i )|$, using the change of variable $x=t-x_{i-1}$, we get
 \begin{equation}\label{eq:fBd}
 \begin{aligned}
 (f, B^d_i ) & = \int_{x_{i-1}}^{x_i} f(t) (B^e_i(t) -B^q_i(t))\, dt \\ 
                 & =  \int_0^h f(x_{i-1}+x) (B^e(x) -B^q(x))\, dx, 
  \end{aligned}
 \end{equation}
 where $B^e$ is defined in \eqref{eq:Be} and $B^q$ is defined in \eqref{eq:Bq} with the special choice of $\beta$ given by  \eqref{eq:magicB}. 
 Next, using  \eqref{eq:Be} and \eqref{eq:Bq}, we write 
 \begin{equation}\label{eq:BemBq}
 B^e(x) -B^q(x) = B^h(x) + B^h_\varepsilon (x), 
 \end{equation}
 where 
 \[
B^h(x)=\lim_{\varepsilon/h \to 0} (B^e(x) -B^q(x)) = 1- \frac{x}{h} - 3 \frac{x}{h} \left ( 1 -  \frac{x}{h} \right ) = \left ( 1 -  \frac{x}{h} \right ) \left ( 1 -  3 \frac{x}{h} \right ), 
 \]
 is independent of $\varepsilon$, and  $B^h_\varepsilon (x)= B^e(x) -B^q(x)- B^h(x)$, i.e., 
  \[
 B^h_\varepsilon (x) = \frac{1 - e^{-\frac{x}{\varepsilon}}}{1 - e^{-\frac{h}{\varepsilon}}} -1 +
 3 \frac{x}{h} \left ( 1 -  \frac{x}{h} \right )\left ( 1- \frac{1 + e^{-\frac{h}{\varepsilon}}}{1 - e^{-\frac{h}{\varepsilon}}}   + \frac{2 \varepsilon}{h} \right).   
 \]
 Using the change of variable $ x=h\, t$ and integration by parts, we have
 \[
   \begin{aligned}
 \int_0^h f(x_{i-1}+x)\,  B^h(x)\, dx &= h  \int_{0}^{1} f(x_{i-1}+h\,t) (1-t)(1-3t)\, dt\\
  &= h \int_{0}^{1} f(x_{i-1}+h\,t) (t-4t^2+t^3)^{'} \, dt \\ 
  &=-h^2 \int_{0}^{1}f'(x_{i-1}+h\,t)(t-4t^2+t^3)\, dt,
    \end{aligned}
 \]
which leads to 
\begin{equation}\label{eq:fBhe}
 \left | \int_0^h f(x_{i-1}+x)\,  B^h(x)\, dx\right | \leq \frac {h^2}{4} \|f'\|_\infty.
 \end{equation}

 For the other integral, using the notation $l_0:= (1 - e^{-\frac{h}{\varepsilon}})^{-1}$ 
 we have 
 \[
 B^h_\varepsilon (x) = l_0\left ( e^{-\frac{h}{\varepsilon}} - e^{-\frac{x}{\varepsilon}} \right )+
 6 \frac{x}{h} \left ( 1 -  \frac{x}{h} \right ) \left ( \frac{\varepsilon}{h} -l_0 e^{-\frac{h}{\varepsilon}} \right ), \ \text {and consequently,} 
 \]
 
 \[
 \begin{aligned}
 \left | \int_0^h f(x_{i-1}+x)\, B^h_\varepsilon (x)\, dx\right | & \leq l_0 \left | \int_0^h f(x_{i-1}+x)\, \left ( e^{-\frac{h}{\varepsilon}} - e^{-\frac{x}{\varepsilon}} \right )\, dx \right | \\ & +  
  6 \left | \int_0^h f(x_{i-1}+x)\, \frac{x}{h} \left ( 1 -  \frac{x}{h} \right ) \left ( \frac{\varepsilon}{h} -l_0 e^{-\frac{h}{\varepsilon}} \right )\, dx\right|.
  \end{aligned}
 \]
The change of variable $ x=h\, t$ in the first integral, leads to 
 \[
  \begin{aligned}
 l_0  \left | \int_0^h f(x_{i-1}+x) \left ( e^{-\frac{h}{\varepsilon}} - e^{-\frac{x}{\varepsilon}} \right )\, dx\right | & =
 l_0 h \left |  \int_{0}^{1} f(x_{i-1}+h\,t)  \left ( e^{-\frac{h}{\varepsilon}} - e^{-\frac{ht}{\varepsilon}} \right )\, dt\right |
 \\ &\leq l_0 h \|f\|_\infty  \int_{0}^{1}   \left ( e^{-\frac{h}{\varepsilon}} - e^{-\frac{ht}{\varepsilon}} \right)dt  \\ &= l_0 h  \|f\|_\infty  \left ( \frac{\varepsilon}{h} \left (1- e^{-\frac{h}{\varepsilon}} \right) -e^{-\frac{h}{\varepsilon}} \right ) \leq  \varepsilon  \|f\|_\infty. 
   \end{aligned}
 \]
It easy to check  that $ \frac{\varepsilon}{h} >l_0 e^{-\frac{h}{\varepsilon}}$. 
Thus,  for the second integral, we obtain 
 \[
   \begin{aligned}
 & 6 \left | \int_0^h f(x_{i-1}+x)\, \frac{x}{h} \left ( 1 -  \frac{x}{h} \right ) \left ( \frac{\varepsilon}{h} -l_0 e^{-\frac{h}{\varepsilon}} \right )\, dx \right | = \\     =  &6\, h   \left ( \frac{\varepsilon}{h} -l_0 e^{-\frac{h}{\varepsilon}} \right ) \left | \int_0^1  f(x_{i-1}+h\,t)   t(1-t)\, dt \right | \leq \varepsilon \|f\|_\infty.
    \end{aligned}
 \]
Combining the above  estimates with \eqref{eq:dif2}, \eqref {eq:fBd}, \eqref{eq:BemBq}, \eqref {eq:fBhe}   leads to \eqref{eq:ConvB2}.
\end{proof} 
  We note here  that for $\varepsilon \leq h^2$, the hypothesis  $e^{-\frac{h}{\varepsilon}} \leq h$ is satisfied. Thus, we have the following corollary of Theorem \ref{th:T}. 
 \begin{corollary}
 Under the hypotheses of Theorem \ref{th:T}, for $\varepsilon \leq h^2$,  we have
 \[
  \max_{j=\overline{1,n-1}}\  \left | u(x_j)-u_j \right |\leq      h^2 \left(6\,  \|f\|_\infty + \frac{3}{4}\,\, \|f'\|_\infty\right )= O(h^2). 
  \]
 
 \end{corollary}
The computations using quadratic bubble UPG discretization with  the choice of  $\beta$ given by \eqref{eq:magicB}, show order $O(h^2)$ or better  for the range of $h$  such that  $\varepsilon \leq h^2$. The order could strictly decrease  for different values $\beta$.

\section{Standard norm approximation for  quadratic bubble UPG } \label{sec:B2norms} 

In this section, we take advantage of the error analysis in the discrete infinity norm of  Theorem \ref{th:T}, and prove error estimates for the quadratic bubble UPG  in the standard $L^2$ norm $\|\cdot\|$  and the standard $H^1$ norm  $|\cdot |$. 

 It is known that the largest eigenvalue of the tridiagonal matrix $S$ as defined in Section \ref{sec:FE} can be bounded above by $4$. Thus, for any \\ $\alpha =[ \alpha_1, \alpha_2, \cdots, \alpha_{n-1} ]^T \in \R^{n-1}$  and $\displaystyle v_h = \sum_{i=1}^{n-1} \alpha_i \varphi_i \in \M_h \subset H^1_0(0,1)$, \\ we have 
 \[
 |v_h|^2 =\frac{1}{h}  \alpha^T S\alpha \leq \frac{4}{h} \sum_{i=1}^{n-1} \alpha_i^2 \leq 12 \, h^{-2} 
 \|v_h\|^2,
 \]
which leads to the norm estimate
\begin{equation}\label{eq:H1L2Inf}
\frac{1}{2\, \sqrt{3}} h\, |v_h| \leq \|v_h\| \leq  \|v_h\|_{h,\infty}, \ \text{for all} \ v_h \in \M_h.  
\end{equation}
In particular, if $u$ is the solution of  the  problem  \eqref{eq:1d-model}, $I_h(u)$ is the linear interpolant of  $u$ on the uniform nodes $x_0, x_1,\cdots, x_n$,  and $\displaystyle u_h$ is the solution of a  bubble UPG discretization  \eqref{eq:1d-modelPG}, 
 then  we  can apply the estimate \eqref{eq:H1L2Inf}  for $v_h= I_h(u) - u_h$ to obtain 
\begin{equation}\label{eq:H1L2Inf2}
\frac{1}{2\, \sqrt{3}} h\, |I_h(u) - u_h| \leq \|I_h(u) - u_h\|_{L^2} \leq  \|I_h(u) - u_h\|_{h,\infty}.
\end{equation}
Next, we state the main result of this subsection.
\begin{theorem} \label{th:T2} For $f\in C^1([0, 1])$  let  $u$  be the solution of  the  problem  \eqref{eq:1d-model}.   Let  $\displaystyle u_h=\sum_{i=1}^{n-1} u_i \varphi_i$ be  the solution of  quadratic bubble UPG discretization  \eqref{eq:1d-modelPG} with  scaling parameter $\beta$ given by \eqref{eq:magicB}. Assume that for  a given $\varepsilon \ll 1$, the mesh size $h$ is chosen such that $\varepsilon \leq h^2$.  Then  
 \begin{equation}\label{eq:ConvB2H1}
|u-u_h| \leq |u-I_h(u)| +  2\sqrt{3}\, h\,   \left (6 \|f\|_\infty + \frac{3}{4} \|f'\|_\infty \right), 
 \end{equation}
 and 
 \begin{equation}\label{eq:ConvB2L2}
\|u-u_h\| \leq \|u-I_h(u)\| +  h^2  \left (6 \|f\|_\infty + \frac{3}{4} \|f'\|_\infty \right). 
 \end{equation}

 \end{theorem}
 \begin{proof} 
 First we note that  $\varepsilon \leq h^2$ implies that  $e^{-\frac{h}{\varepsilon}} \leq h$. Thus, the inequality  assumption of Theorem \ref{th:T} is satisfied.  For justifying  \eqref{eq:ConvB2H1}, we use the triangle inequality in the energy norm  for 
 \[
 u-u_h= (u-I_h(u)) + (I_h(u) -u_h),
 \]
 and note that 
 \[
 \|I_h(u) - u_h\|_{h,\infty}= \max_{j=\overline{1,n-1}}\  \left | u(x_j)-u_j \right |.
 \]
Now, the estimate  \eqref{eq:ConvB2H1} is a direct consequence of \eqref{eq:H1L2Inf2} and Theorem  \ref{th:T}. 
 Similar arguments can be used to justify \eqref{eq:ConvB2L2}. 
 \end{proof}
\begin{remark}
The estimates \eqref{eq:ConvB2H1} and  \eqref{eq:ConvB2L2} imply 
 $O(h)$ approximation  for the energy norm and $O(h^2)$ approximation  for the $L^2$ norm respectively, provided the 
interpolant $I_h(u)$ approximates the exact solution $u$ with the same corresponding order.
Due to the possible  presence of a  boundary layer for  the solution, standard  interpolant  approximation  happens only on subdomains away from the boundary layer. 

As an example of posible  large magnitude for  $|u-I_h(u)|$  on subdomains containing the boundary layer, we consider $f=1$ in  \eqref{eq:1d-model}  with  the exact solution  $u$ 
\[
{u(x)= x- \frac{e^{\frac{x}{\varepsilon}} -1}{e^{\frac{1}{\varepsilon}} -1}}. 
\]
Straightforward  calculations give 
\[
\begin{aligned}
{|u- I_h(u)|_{[0, 1]}^2} &=  \frac{1+e^{-1/\varepsilon}} {1-e^{-1/\varepsilon}} 
\left (\frac{1}{2 \varepsilon} - \frac{1}{h} \, \frac{1-e^{-h/\varepsilon}}{1+e^{-h/\varepsilon}} \right )\ \text{and}\\
|u- I_h(u)|_{[0, 1-h]}^2 &=  \frac{e^{-2h/\varepsilon}-e^{-2/\varepsilon}}{1-e^{-2/\varepsilon}} \, 
|u- I_h(u)|_{[0, 1]}^2.
\end{aligned}
\]
Thus, for $\varepsilon \ll h$, we have
\[
\begin{aligned}
|u- I_h(u)|_{[0, 1-h]} & \approx e^{-h/\varepsilon} |u- I_h(u)|_{[0, 1]} \approx 0, \ \text{and} \\
|u- I_h(u)|_{[1-h, 1]} & \approx |u- I_h(u)|_{[0, 1]} \approx \frac{1}{2\varepsilon} -\frac{1}{h}. 
\end{aligned}
\]
This calculations show that the energy error  $|u-I_h(u)|_{[0, 1]}$ 
is insignificant for the  interval $[0, 1-h]$, and it could be very large and  essentially attained  on the last sub-interval $[1-h, 1]$. Thus, in this case,  the interval  $[0, 1-h]$ qualifies  as an ``away from the boundary layer'' subdomain for the energy error $|u-u_h|$. Numerical tests show that $|u-u_h|_{[0, 1-h]} \leq O(h)$ for   $\varepsilon \leq h^2$. 
\end{remark} 

\section {Two dimensional   bubble UPG } \label{sec:2DB2} 

In this section we extend the bubble UPG approach to  the two dimensional  case of problem  \eqref{PDE_RD}.   Even though the ideas  presented in this section can be implemented  to the case of a  general bounded domain and   an arbitrary convection vector  $\textbf{b}=(b_1, b_2)^T \neq 0$, to simplify our presentation, we will assume that  $\Omega=(0,1)\times (0, 1) $  and  $\textbf{b}=[1,0]^T$. With these assumptions, the problem \eqref{PDE_RD} becomes: Given $f \in L^2(\Omega)$, find $u=u(x,y)$ such that 
 \begin{equation}\label{PDEx}
   \left\{
\begin{array}{rcl}
     -{ \varepsilon}\, \Delta u +  u_x & =\ f & \mbox{in} \ \ \ \Omega,\\
      u & =\ 0 & \mbox{on} \ \partial\Omega.\\ 
\end{array} 
\right.
 \end{equation}

In this section, we  consider a natural extension of the  bubble UPG method  described in  Section \ref{sec:FE} for discretizing \eqref{PDEx}. We start by dividing the $x$-interval $[0,1]$ into $n$  equal length subintervals using the nodes \\ $0=x_0<x_1<\cdots < x_n=1$. Similarly, we divide the  $y$-interval $[0,1]$ into $n$  equal length subintervals using the nodes $0=y_0<y_1<\cdots < y_n=1$. We  denote  $h=x_j - x_{j-1}=y_j - y_{j-1}=1/ n$.

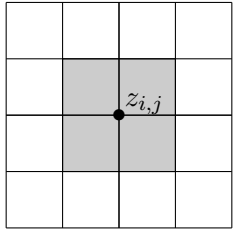
\begin{figure}[!h]\label{fig:exmesh}
	\begin{center}
		\begin{tikzpicture}[scale=0.75]
						\draw[fill=black!20!white] (1,1)--(3,1)--(3,2)--(3,3)--(1,3)--cycle;
			\draw[](0,0)--(4,0);
			\draw[](4,4)--(4,0);
			\draw[](0,4)--(4,4);
			\draw[](0,0)--(0,4);
			\draw[](1,0)--(1,4);
			\draw[](2,0)--(2,4);
			\draw[](2,0)--(2,4);
			\draw[](0,2)--(4,2);
			\draw[](3,0)--(3,4);
			\draw[](0,1)--(4,1);
			\draw[](0,2)--(4,2);
			\draw[](0,3)--(4,3);
			\node[] at (2,2) {$\bullet$};
			\node[] at (2.45,2.2) {$z_{i,j}$};
			
		\end{tikzpicture}
	\end{center}
\end{figure}

Using  the notation  of Section \ref{sec:FE},  we define  the trial space $\M_h$ by 
\[
\M_h= \text{span}\{\varphi_i(x)\,\varphi_j(y) \} \ \text{for all} \ z_{i,j}=(x_i,y_j) \in  \Omega,  
\]
and  the test space $V_h$  by 
\[
V_h= \text{span}\{ g_i(x) \,\varphi_j(y) \} \ \text{for all} \ z_{i,j}=(x_i,y_j) \in  \Omega.  
\]
As defined in Section \ref{sec:FE}, $g_i(x)= \varphi_i(x)+ B_i(x)-B_{i+1}(x)$,  
where $B:[0,h] \to\R$ is a bubble function satisfying  \eqref{Bbounds}, \eqref{b1}, and  $B_i(x)=B(x-x_{i-1})$ on $[x_{i-1}, x_i]$ and is extended by zero on the entire $[0, 1]$. 

A general  Upwinding Petrov Galerkin discretization with  bubble functions in the $x$-direction for solving 
\eqref{PDEx} is: Find $u_h \in \M_h$ such that 
\begin{equation}\label{eq:UPGxV}
b(v_h, u_h):= \varepsilon\, (\nabla u_h, \nabla v_h)+  \left (\frac{\partial u_h}{\partial x},v_h\right ) =(f,v_h) \ \Forall  v_h \in V_h. 
\end{equation}
If the average value $b$ of the generating bubble $B$ is not too large, then  the existence and the uniqueness of the solution of \eqref{eq:UPGxV} can be proved by investigating the  corresponding linear system as shown in the next section for the quadratic bubble UPG.  
\subsection{2D  quadratic  bubble UPG } \label{subsec:2DB2x} 
In this section,  we focus on the  bubble UPG  method  with quadratic bubble functions in the $x$-direction. In the general UPG discretization, we choose 
\[
 B (x)= B^q(x)= \frac{4\, \beta}{h^2} x(h-x), \ \text{with} \  \beta =\frac32 \left (\frac{1}{2 g_0}- \frac{\varepsilon}{h}\right ).
\]
For describing the matrix of the linear system associated with quadratic bubble discretization,  
we will need the matrices introduced in Section \ref{sec:FE}:  \(S,M^e_{fe}  \in\R^{n-1}\times \R^{n-1}\),  where 
\[
S=tridiag(-1, 2, -1) \ \text{and} \   M^e_{fe}= tridiag \left ( -\frac{1+t_e}{2t_e} ,\  \frac{1}{t_e},\   -\frac{1-t_e}{2t_e} \right ).
 \]
 According to   Remark \ref{re:magicB}, we have  $M^q_{fe} =M^e_{fe}$. To simplify the notation, we will define  $C^e:=M^e_{fe}=M^q_{fe}$. 
We will also need   the matrices that correspond to the one dimensional mass matrices $M$ and $M^q$ with entries 
\[
M_{ij}=(\varphi_j,\varphi_i)\  \text{and} \  M^q_{ij}= (\varphi_j, g^q_i), i,j=1,2,\cdots,n-1.
\]
Simple calculations show that
\[
M=\frac{h}{6} \, tridiag(1, 4, 1),   \ \text{and} \  M^q = M+\beta\, \frac{ h}{3} \, tridiag(-1, 0,1).
 \]
By expanding the solution $u_h$ of   \eqref{eq:UPGxV}  as 
\[
u_h=\sum_{k=1}^{n-1}\sum_{l=1}^{n-1} u_{lk}\,  \varphi_l(x)\,  \varphi_k(y),
\]
and by  taking  $v_h= g^q_i(x) \,\varphi_j(y)$ in \eqref{PDEx}, we get the linear system:  
\begin{equation}\label{eq:sys2dq}
A^q\, U^q = F^q, \ \text{where} 
\end{equation}
 \begin{equation}\label{eq:matrix2dq}
A^q = M\, \otimes C^e + \frac{\varepsilon}{h}\, S \otimes M^q, \ \text{and}
\end{equation}
 \begin{equation}\label{eq:DualV2dq}
\begin{aligned}
U^q&=[u_{11}, \cdots, u_{n-1,1}, u_{12},u_{22}, \cdots,u_{n-1,2},\cdots, u_{n-1,n-1}]^T,\\
F^q&=[(f,g^q_1\,\varphi_1),\cdots, (f,g^q_{n-1}\, \varphi_1), (f,g^q_1\, \varphi_2), \cdots, (f,g^q_{n-1}
\, \varphi_{n-1})  ]^T.
\end{aligned}
\end{equation}
The equations \eqref{eq:sys2dq}-\eqref{eq:DualV2dq}  lead to a fast implementation for finding $u_h$.

\begin{remark}\label{rem:AqA0}
 Using that  $t_e:=\tanh\left (\frac{h}{2\varepsilon}\right)$ and $\beta$ is given by \eqref{eq:magicB}, we have 
 \[
  \lim_{\frac{\varepsilon}{h}\to 0}   \,       t_e = 1, \ \text {and}\   \lim_{\frac{\varepsilon}{h}\to 0} \, \beta=  \frac{3}{4},  
  \]
  and the limits are exponentially fast.   Consequently,  the following  limits hold exponentially fast as well:
 \[
 C^e \to C^0:=tridiag(-1, 1,0),  \ M^q\to M^{q_0}:=\frac{h}{12}tridiag(-1, 8,5).
 \]
 Thus, for  $\varepsilon  \ll h$, we note that the matrix $A^q$ is ``exponentially close'' to 
 \begin{equation}\label{eq:AqA0}
 \begin{aligned}
 h\,  &\left[(1/6) \, tridiag(1, 4, 1)  \otimes tridiag(-1, 1,0) \right]  \\
 +  \varepsilon\, & \left[ \ tridiag(-1, 2,-1) \  \otimes \ tridiag(-1, 8,5)  \right].  
 \end{aligned}
\end{equation}
In addition, for $\varepsilon  \ll h$,  we also have that $A^q$ is very close to 
\[
 h \, \left[(1/6) \, tridiag(1, 4, 1)  \otimes tridiag(-1, 1,0) \right],  
\]
which is an invertible matrix because  both matrices, $tridiag(1, 4, 1)$  and \\ $ tridiag(-1, 1,0)$, are invertible. Thus, $A^q$ is an invertible matrix  and the second term in formula  \eqref{eq:matrix2dq} for $A^q$  is less significant in approximating  $u_h$. 

\end{remark} 
While a precise convergence analysis of the method as done in the one dimensional case might be difficult, Remark \ref{rem:AqA0} suggests that  for $\frac{\varepsilon}{h}\to 0$, the contribution of the  
 $\displaystyle \varepsilon\, \left (\frac{\partial u_h}{\partial y}, \frac{\partial v_h}{\partial y}\right )$ in the variational formulation \eqref{eq:UPGxV},  is not essential and, at least for $\varepsilon  \ll h$, the discrete infinity error should behave as in the one dimensional case. We further notice that norm estimates similar to \eqref{eq:H1L2Inf} hold true with different constants for the two dimensional case. Thus, the behavior of the $L^2$ and the $H^1$ errors for the 2D case should mirror the 1D case. This was  observed indeed in our numerical tests. We present  two numerical examples in the next section. 

\section{Numerical Results}  \label{sec:2DB2NR} 
In this section, we  present  numerical conclusions  for discretizing \eqref{PDEx} using the  quadratic bubble UPG  approximation of Section \ref{subsec:2DB2x}. 
\vspace{0.1in}

{\bf Example 1:  Elliptic boundary layer near $x=1$.}  For this example, we choose the right hand side $f$ such that the exact solution is 
\[
u(x,y) = v(x) \, w(y), \ \text{with} \ w(y)=\sin(\pi y),  \text{and} 
\]
\[
v(x)= \frac{1}{1-\varepsilon} \left (  e^x -e - \frac{e-1}{1-e^{-1/\varepsilon}} \left ( e^{\frac{x-1}{\varepsilon}} -1  \right ) \right). 
\]
We note that $v$ is the unique solution of 
\[
\begin{cases} 
 -\varepsilon \, v''(x) + v'(x) =e^x, \\ v(0)=v(1)=0,
 \end{cases}
\]
and has  a boundary layer near $x=1$. 

We approximated the exact solution   for various  values of $\varepsilon \leq 10^{-6}$, and  $h=\frac{1}{2^n}, n=5,6,7,8,9,10$.  Note  that in all these cases,  $\varepsilon \leq h^2$. For all values  of $\varepsilon$ and the specified values of $h$, we verified that 
\begin{equation}\label{eq:orderInterpInf}
\|u_h -I_h(u)\|_{h,\infty} = \mathcal{O}(h^2). 
\end{equation}
We also measured  the $L^2$ and the $H^1$ errors and observed that
\[
|u-u_h| = \mathcal{O}(h) , \text{and} \  \|u-u_h\|_{L^2}  =\mathcal{O}(h^2)
\]
on the subdomain $(0, 1-\delta)\times (0, 1) $ for $\delta=0.01-$ away form the boundary layer.
By  decreasing $\delta$, for example, to  $\delta=0.001$, we still observe that \eqref{eq:orderInterpInf} holds, 
  but the  $H^1$ and the  $L^2$ errors  increase  and  {\it  their orders of convergence decrease} when compared with the  $\delta=0.01$ case. However,  for $\delta=0.001$, on  $(0, 1-\delta)\times (0, 1) $, we have that
\[
{ |u-u_h| \approx  | I_h(u)-u_h| } , \text{and} \ {\|u-u_h\|_{L^2}  \approx \|I_h(u)-u_h\|_{L^2} }.
\]
\vspace{0.05in}

In conclusion, the loss in convergence order  when the errors are measured on the whole domain, is due to suboptimal  approximation  of  the interpolant near the boundary layer of the exact solution, and is not due to a weakness of the quadratic bubble UPG discretization. 
 
\vspace{0.1in}

{\bf Example 2: Elliptic boundary layer near $x=1$ and parabolic  near boundary layer near 
$y=0$ and $y=1$.} We choose the right hand side $f$ such that the solution 
$u(x,y) = v(x) \, w(y)$ with  $v$ as  defined in Example 1, and $\displaystyle w(y)=y(1-y)+e^{-\frac{y}{\sqrt{\varepsilon}}} + e^{-\frac{1-y}{\sqrt{\varepsilon}}} $. 
Even though the exact solution exhibits both elliptic and parabolic boundary layers, we have that for 
 $\varepsilon \leq  h^2$, the estimate \eqref {eq:orderInterpInf} still holds, away from the parabolic boundary layers. In addition, the  $H^1$ and $L^2$  errors match the order of the interpolant approximation, away from both types of  boundary layers.  
 For all cases when  $h^2 \leq  \varepsilon$, the numerical tests  showed  that the numerical solution is ``identical'' in the {\it eye ball measure} with the exact solution. On the other hand,  for values of  $\varepsilon$  and $h$, such tat $\varepsilon < h^2$, the discrete solution exhibits no-physical oscillation as the plot of the numerical  solution drops along  the parabolic  boundary layers, near  $y=0$ and $y=1$.  We further noted  that  for  $ h^2 \approx \varepsilon$,  the discrete solution is free of no-physical oscillation and approximates well the exact solution in both   $H^1$ and $L^2$ norms, {\it away from all boundary layers}. 


\section{Conclusion} \label{sec:Conc} 
We analyzed  a bubble upwinding Petrov-Galerkin discretization method for the convection diffusion problem.  For  the one dimensional case with a special scalling parameter for the quadratic bubble UPG, we proved an optimal convergence estimate in the discrete infinity norm. As a consequence, we obtain optimal error estimates in the standard  $H^1$ and $L^2$  norms away from the boundary layers. The approach was extended to a special two dimensional case. 
The main advantage  of the proposed  bubble upwinding approach is that, by using uniform meshes,  one can recover optimal or near optimal error estimates  for the discrete solutions in standard norms. Due to optimal approximation in the discrete infinity norm, the discrete solutions  are free of  non-physical oscillations.  
The loss in convergence order for the  $H^1$ and $L^2$ errors, computed on the entire domain, is due to  suboptimal  approximation properties of  the interpolant of the exact solution on uniform  meshes. This convergence  aspect  is not a weakness of the proposed  discretization, and can lead to building more efficient  bubble UPG methods on non-uniform meshes designed to optimize the interpolant approximation. 

New designed discretizations of multi-dimensional convection dominated problems, could take
advantage of the  efficient discretizations of the 1D  and the 2D problems presented here. 
The main take away of our results is that, for building robust discretizations of convection dominated problems, one efficient strategy  is to tensor an efficient  bubble UPG discretization along each  stream line with standard discretizations on the “orthogonal” direction(s).
  

 \bibliographystyle{plain} 
\def\cprime{$'$} \def\ocirc#1{\ifmmode\setbox0=\hbox{$#1$}\dimen0=\ht0
  \advance\dimen0 by1pt\rlap{\hbox to\wd0{\hss\raise\dimen0
  \hbox{\hskip.2em$\scriptscriptstyle\circ$}\hss}}#1\else {\accent"17 #1}\fi}
  \def\cprime{$'$} \def\ocirc#1{\ifmmode\setbox0=\hbox{$#1$}\dimen0=\ht0
  \advance\dimen0 by1pt\rlap{\hbox to\wd0{\hss\raise\dimen0
  \hbox{\hskip.2em$\scriptscriptstyle\circ$}\hss}}#1\else {\accent"17 #1}\fi}

 \end{document}